\documentclass{amsart}

\DeclareMathSymbol{\twoheadrightarrow}  {\mathrel}{AMSa}{"10}
        \def\GG{{\mathcal G}}

\def\Q{{\mathbb Q}}
\def\Z{{\mathbb Z}}
\def\C{{\mathbb C}}

\def\W{{\mathbb W}}

\def\D{{\mathbb D}}

\def\RR{{\mathbb R}}
\def\F{{\mathbb F}}
\def\P{{\mathbb P}}

\def\f{{\tilde F}}
                     \def\f0{{\mathfrak f}}

             \def\K{\mathrm{K}}

\def\A8{{\mathbf A}_8}
\def\Alt{\mathrm{Alt}}

\def\RR{{\mathfrak R}}
\def\Perm{\mathrm{Perm}}
\def\Gal{\mathrm{Gal}}

\def\End{\mathrm{End}}
\def\Aut{\mathrm{Aut}}

                  \def\cl{\mathrm{cl}}

\def\alb{\mathrm{alb}}

\def\ST{{\mathbf S}}

\def\AT{\mathbf{A}}

\def\dim{\mathrm{dim}}
\def\Oc{{\mathcal O}}

                            \def\Prym{\mathrm{Prym}}

\newtheorem{thm}{Theorem}[section]
\newtheorem{lem}[thm]{Lemma}

\theoremstyle{definition}
\newtheorem{defn}[thm]{Definition}
\newtheorem{ex}[thm]{Example}
\newtheorem{rem}[thm]{Remark}

\newtheorem{rems}[thm]{Remarks}
        \newtheorem{sect}[thm]{}
\begin{document}

\title[Prym varieties]{Prym varieties that are not isomorphic to Jacobians}

\author[Yuri G.\ Zarhin]{Yuri G.\ Zarhin}
\address{Department of Mathematics, Pennsylvania State University,
University Park, PA 16802, USA}

\email{zarhin\char`\@math.psu.edu}

\thanks{The  author was partially supported by  the Travel Support for Mathematicians Grant MPS-TSM-00007756 from the Simons Foundation. Most of this work was done in August-September 2025 during his stay at Max-Planck-Institut f\"ur Mathematik (Bonn, Germany), whose hospitality and support are gratefully acknowledged.}

\begin{abstract}
We study Prym varieties of ramified (at precisely two points) double covers of smooth irreducible complex projectives curves that admit an automorphism
of  prime order $p>2$. Using Galois  theory, we give an explicit construction of Prym varieties that are not isomorphic to jacobians
(even if one ignores the polarizations).

\end{abstract}


\maketitle

\section{Introduction}
In this paper we study   
complex abelian varieties $X$ 
 endowed with a ring homomorphism
\begin{equation}
\label{embed}
\Z[\zeta_p] \hookrightarrow \End(X), \quad 1 \mapsto 1_X
\end{equation}
where  $p$ is an odd prime, $\zeta_p=\exp(2\pi \mathbf{i}/p)\in \C$ is a primitive $p$th root of unity in the field $\C$ of complex numbers,
 $\Z[\zeta_p]$ is the $p$th cyclotomic ring, and $1_X$ is the identity automorphism  of $X$. 
Using Galois theory and the theory of Prym varieties, we consruct explicitly such abelian varieties that are not isomorphic to jacobians of curves.

In order to explain our approach, let us start with the following notation and definitions.
As usual, $\Z$ and $\Q$ denote the ring of integers and the field of rational numbers,
and $\Q(\zeta_p)$  is the $p$th cyclotomic field.
We write $\lambda$ for
the (principal) maximal ideal $(1-\zeta_p)\cdot\Z[\zeta_p]$ of $\Z[\zeta_p]$.

Let $f(x)\in\C[x]$ be a polynomial of degree $n\ge 4$ without repeated roots.
We assume that $p$ does {\sl not} divide $n$.
Let $C_{f,p}$ be a smooth projective model of the smooth plane  affine curve
$y^p=f(x)$. It is well known (\cite{Koo}, pp. 401-402, \cite{Towse}, Prop. 1 on
p. 3359, \cite{Poonen}, p. 148) that the genus $g(C_{f,p})$  of $C_{f,p}$ is
$(n-1)(p-1/2$.

The map $(x,y)\mapsto (x,\zeta_p y)$ gives rise to a non-trivial birational
automorphism $\delta_p:C_{f,p}\to C_{f,p}$ of period $p$. By functoriality,
$\delta_p$ induces the linear operator in the space of differentials of the
first kind
$$\delta_p^{*}: \Omega^1(C_{f,p}) \to \Omega^1(C_{f,p}).$$
Its spectrum consists of (not necessarily all) primitive $p$th roots of unity
(eigenvalues) $\zeta_p^{-i}$ ($1 \le i \le p-1$);
the multiplicity of $\zeta_p^{-i}$ is  $[ni/p]$
\cite{ZarhinCamb}.

 Let $J(C_{f,p})$ be the jacobian of $C_{f,p}$; it is an  abelian
variety of dimension $g(C_{f,p})$ that carries the canonical principal polarization. We write $\End(J(C_{f,p}))$ for
the ring of endomorphisms of $J(C_{f,p})$. By Albanese functoriality, $\delta_p$
induces an automorphism of $J(C_{f,p})$ which we still denote by $\delta_p$; it
is known ([11], p. 149, [14], p. 448) that $\sum_{i=0}^{p-1}\delta_p^i=0$ in
$\End(J(C_{f,p}))$. This gives us the ring embedding
$$\Z[\zeta_p]\cong \Z[\delta_p]\subset \End(J(C_{f,p})), \ \zeta_p \mapsto \delta_p,$$
which sends $1$ to the identity automorphism of $J(C_{f,p})$
(\cite[p. 149]{Poonen}, [14], \cite[p. 448]{Schaefer}). The canonical principal polarization on $J(C_{f,p})$
is $\delta_p$-invariant.

 Let $K$ be a subfield of $\C$ that contains $\zeta_p$ and all
coefficients of $f(x)$, i.e.,
$$\Q(\zeta_p)\subset K, \quad f(x)\in K[x]\subset \C[x].$$
Let $\bar{K}$ be the algebraic closure of $K$ in  $\C$ and 
$\RR_f\subset \bar{K}\subset \C$ be the set of roots of $f(x)$ and $K(\RR_f)$ the
splitting field of $f(x)$ over $K$. Clearly, $K(\RR_f)$ is a finite Galois
extension of $K$. We write $\Gal(f)$ for the (finite) Galois group
$\Gal(K(\RR_f)/K)$. One may view $\Gal(f)$ as a certain permutation subgroup of
the group $\Perm(\RR_f)$ of  permutations of $\RR_f$. If we (choose an order on
$\RR_f$, i.e.,) denote  the roots of $f(x)$ by $\{\alpha_1, \dots , \alpha_n\}$
then we get a group isomorphism between $\Perm(\RR_f)$ and the full symmetric
group $\ST_n$ and $\Gal(f)$ becomes a certain subgroup of $\ST_n$.

It is proven in \cite{ZarhinCamb,ZarhinA5} that if  $ n \ge 5$ and $\Gal(f)$ is either $\ST_n$, or $\AT_n$ then
$$\End(J(C_{f,p}))=\Z[\delta_p]\cong \Z[\zeta_p].$$

Now, in order to deal with Prym varieties,  we will consider appropriate cases when the curve
$C_{f,p}$ is equipped with an involution.
let us assume that  $f(x)$ is an {\sl odd} polynomial of {\sl odd} degree $n=2m+1$.
Then $f(x)=x h(x)$ where $h(x)$ is an even polynomial of degree $2m$ with nonzero constant term
(recall that $f(x)$ has no repeated roots).
We will be mostly  interested in the situation when $\Gal(f)$ coincides with (or contains) the Weyl group
$\W(\D_m)$ of the root system
 $\D_m$ in the following sense. Since $f(x)$ is  odd and without repeated roots, there exist
$m$ distinct non-zero roots $\{\beta_1, \dots , \beta_m\}$ of $f(x)$ such that
($\beta_i \ne \pm \beta_j$ if $i\ne j$ and) 
$$\RR_h=\{\pm \beta_1, \dots , \pm \beta_m
\} 
\subset \{0\} \cup \{\pm \beta_1, \dots , \pm \beta_m
\}=\RR_f \subset \bar{K}.$$ Then $\Gal(f)$ may be viewed as a certain
group of  permutations  of $\RR_f$ of the form
 $$(\epsilon,s): 0 \mapsto 0, \ \beta_i \mapsto \epsilon(i) \beta_{s(i)}, \ -\beta_i \mapsto -\epsilon(i) \beta_{s(i)}$$
 where $s \in \ST_m$ is an arbitrary permutation of $\{1,2, \dots,m\}$ and 
is  a
 ``sign function'' $\epsilon: \{1,2, \dots,m\} \to \{\pm 1\}$.
 We define    $\W(\D_m)$ as the group of all permutations $(\epsilon,s)$ of $\RR_f$ such that
 $s \in \ST_m$ and the sign function $\epsilon$ 
satisfies the condition $\prod_{i=1}^m \epsilon(i) =1$. 

The regular map of curves
$$\pi:C_{f,p}\to \P^1, \ (x,y)\mapsto x$$
has degree $p$ and branches exactly at $0$, the $2m$-element set 
$\{\alpha\mid \alpha\in \RR_h\}$ and $\infty$. 
 Since $p$ is an odd integer and $f(x)$ is an even polynomial, the curve $C_{f,p}$ admits the
involution
$$\delta_2:C_{f,p} \to C_{f,p}, \quad (x,y) \mapsto (-x,-y),$$
which commutes with $\delta_p$.
The automorphism $\delta_2:C_{f,p} \to C_{f,p}$ has exactly two fixed points,
namely, $\pi^{-1}(0)$ and $\pi^{-1}(\infty)$. It is well known that the
quotient $\tilde{C}_{f,p}=C_{f,p}/(1,\delta_2)$ is a smooth (irreducible)
projective curve; $C_{f,p}\to \tilde{C}_{f,p}$ is a double covering that is
ramified at exactly two points, namely  $\pi^{-1}(0)$ and
$\pi^{-1}(\infty)$. The Hurwitz's formula implies that the genus of
$\tilde{C}_{f,p}$ is $m(p-1)/2$.  Clearly, all  ramification points of $\pi$ in $C_{f,p}$ are
$\delta_p$-invariant. By abuse of notation, we denote $\pi^{-1}(\infty)$ by
$\infty$. Let us put
$$B=\pi^{-1}(\RR_h)=\{(\alpha,0)\mid \alpha\in \RR_h\}\subset C_{f,p}(\bar{K}).$$
Clearly, all elements of $B$ are $\delta_p$-invariant. On the other hand, if
$P= (\alpha,0)\in B$ then $\delta_2(P)=(-\alpha,0)\in B$.

By Albanese functoriality, $\delta_2$ induces the
automorphism of $J(C_{f,p})$ which we still denote by $\delta_2$ and which
(still) commutes with the automorphism $\delta_p$ of $J(C_{f,p})$. We have
$\delta_2^2=1$ in $\End(J(C_{f,p}))$.

By definition \cite{MumfordP}, the Prym variety (aka prymian) of the double cover $C_{f,p}\to \tilde{C}_{f,p}$
is 
$$\mathrm{Prym}(C_{f,p}):=(1-\delta_2)J(C_{f,p}).$$
Since $\delta_p$ commutes with $\delta_2$, $\mathrm{Prym}(C_{f,p})$ is a $\delta_p$-invariant
abelian subvariety of $J(C_{f,p})$, which gives rise to the ring homomorphism
\begin{equation}
\label{PrymP}
\Z[\zeta_p] \to \End(\mathrm{Prym}(C_{f,p})),
\end{equation}
which sends $1$ to the identity automorphism of $\mathrm{Prym}(C_{f,p})$. This implies
easily that the map \eqref{PrymP} is a ring embedding; we denote by $\delta_{p,-}$
the image of $\zeta_p$ in $\End(\mathrm{Prym}(C_{f,p}))$. Clearly, $\delta_{p,-}$
satisfies the $p$th cyclotomic equation in $ \End(\mathrm{Prym}(C_{f,p}))$.
It is known that 
$$\dim(\mathrm{Prym}(C_{f,p})))=g((C_{f,p}))-g( \tilde{C}_{f,p})=$$
$$(n-1)(p-1)/2 -m(p-1)/2=
2m(p-1)/2 -m(p-1)/2=m(p-1)/2.$$ 
Since $C_{f,p}\to \tilde{C}_{f,p}$ is ramified at precisely two points, it follows from a
theorem of Mumford \cite{MumfordP} that the restriction of the principal polarization on 
$J(C_{f,p})$ to $\mathrm{Prym}(C_{f,p}))$ is twice a principal polarization on $\mathrm{Prym}(C_{f,p}))$.
This gives us the canonical principal polarization on $\mathrm{Prym}(C_{f,p}))$, which is obviously
$\delta_{p,-}$-invariant.

We still need one more definition. Let us
consider the $m$-dimensional $\F_p$-vector space of {\sl odd} functions
$$V_f^{-}:=\{\phi:\RR_f \to \F_p\mid \phi(-\alpha)=-\phi(\alpha) \ \forall
\alpha\in \RR_f\}$$ endowed with the natural structure of a $\Gal(K)$-module.

Our main result is the following, which in the case $m \equiv -1 (\bmod p)$
gives an explicit description of the spectrum of $\delta_{p,-}$ with respect to its
action on the differentials of the first kind on the prymian. In addition, if $\Gal(f)$ coincides with
$\W(\D_m)$  then we prove that $\End( \mathrm{Prym}(C_{f,p}))$ coincides with
$\Z[\delta_{p,-}]\cong \Z[\zeta_p]$.
 (The case $p=3$ was done earlier in \cite{ZarhinI}.)

\begin{thm}
\label{main0} Suppose that $r\ge 2$ is an integer and let us put
$$m:=pr-1,  \quad n=2m+1=2pr-1\ge 5.$$
Let 
$f(x)\in \C[x]$ be an odd polynomial of degree $n=2m+1$ without repeated roots. Then:

\begin{itemize}
\item[(i)]
\begin{enumerate}
\item

The principally
polarized abelian variety $\mathrm{Prym}(C_{f,p})$ is not isomorphic to the canonically
polarized jacobian of any smooth irreducible projective curve of genus $m(p-1)/2$.

\item By functoriality, $\delta_{p,-}$ induces the linear operator
$$\delta_{p,-}^{*}: \Omega^1(\mathrm{Prym}(C_{f,p})) \to \Omega^1(\mathrm{Prym}(C_{f,p}))$$
in the space of differentials of the first kind on $\mathrm{Prym}(C_{f,p})$. 
 The spectrum of $\delta_{p,-}^{*}$ consists of all primitive $p$th roots of unity $\zeta_p^{-j}$ ($1 \le j \le p-1$). The multiplicity of eigenvalues
 $\zeta_p^{-j}$ is $rj$ if $j$ is even and $rj-1$ if $j$ is odd.
 In particular, all these multiplicities are relatively prime if and only if 
 $r$ is even.

\end{enumerate}
 \item[(ii)] Suppose that $K$ is a subfield
of $\C$ that contains $\zeta_p$ and all coefficients of $f(x)$. Then:
\begin{itemize}
\item[(a)]
 The abelian variety $\mathrm{Prym}(C_{f,p})$ and its automorphism $\delta_{p,-}$ are defined over
$K$. In addition, the Galois submodule $\mathrm{Prym}(C_{f,p})^{\delta_{p,-}}$ of
$\delta_{p,-}$-invariants of $\mathrm{Prym}(C_{f,p})(\bar{K})$ is canonically isomorphic to
 $V_f^{-}$.
  \item[(b)]
 Assume
additionally that  $r$ is even and
$\Gal(f)$ coincides with $\W(\D_m)$. Then:
\begin{itemize}
\item[(b1)] $\End(\mathrm{Prym}(C_{f,p}))=\Z[\delta_{p,-}] \cong
\Z[\zeta_p]$. In particular, $\mathrm{Prym}(C_{f,p})$ is an
absolutely simple abelian variety.
 \item[(b2)] The abelian variety
$\mathrm{Prym}(C_{f,p})$ is isomorphic neither to the jacobian of a connected smooth projective curve
nor to a product of jacobians of smooth projective curves (even if one ignores
polarizations).
\end{itemize}
\end{itemize}
\end{itemize}
\end{thm}

\begin{rem}
A complete list of those (generalized) Prym varieties that are isomorphic, as
canonically principally polarized abelian varieties, to jacobians of smooth projective
curves or to products of them was given by V.V. Shokurov
\cite{SlavaInv,SlavaIzv}. In the course of the proof of Theorem
\ref{main0}(i)(1) we use a different approach based on the study of the action of
the period $p$ automorphism on the differentials of the first kind
\cite{ZarhinMF,ZarhinAB}. 
\end{rem}

\begin{rem}
\label{scheme}
The  proof  of Theorem \ref {main0} is based on a combination of ideas from our previous papers.Let us explain briefly how various conditions in Theorem \ref{main0}  are used.
Mathods of representation theory  \cite{ZarhinKumar} allow us to prove that  the ``largeness'' of $\Gal(f)$ combined with Galois properties of the torsion points
of the prymian  implies that $\Z[\delta_{p,-}] $ coincides with its own centralizer
in $\End(\mathrm{Prym}(C_{f,p}))$ and therefore $\Q[\delta_{p,-}] $ coincides with its own centralizer
in $\End^{0}(\mathrm{Prym}(C_{f,p}))$. In particular, $\Q[\delta_{p,-}]\cong \Q(\zeta_p) $ contains the center of $\End^{0}(\mathrm{Prym}(C_{f,p}))$.
On the other hand, the multiplicity properties of
eigenvalues of  $\delta_{p,-}^{*}$  combined with basic properties  of the Hodge group of
$\mathrm{Prym}(C_{f,p}))$ \cite{ZarhinCamb}
imply that $\Q[\delta_{p,-}] $  cannot contain the center of $\End^{0}(\mathrm{Prym}(C_{f,p}))$
as a proper subalgebra.  This implies that  $Q[\delta_{p,-}] $ is the center that coincides with its own centralizer, which means that $Q[\delta_{p,-}] =\End^{0}(\mathrm{Prym}(C_{f,p}))$.  On the other hand,  the explicit information about multiplicities allows us to conclude (thanks to results of
\cite{ZarhinAB})
 that our prymian is not isomorphic to any jacobian.

\end{rem}

The paper is organized as follows. In Section \ref{group} we discuss permutation modules
 related to Galois groups of odd polynomials. In Section \ref{geom} we study superelliptic
 jacobians and prymians and prove the main result. In Appendix (Section \ref{appen})
 we compute the Galois groups of certain polynomials that were used in order to construct
 explicit examples of Prym varieties that are not isomorphic to jacobians.

{\bf Acknowledgements}. I am grateful to V.V. Shokurov and A.E. Zalesski  for useful discussions. 
My special thanks go to the referee, whose comments helped to improve the exposition.

\section{Galois groups of odd polynomials and permutation modules}
\label{group}

Let $K$ be a field of characteristic zero, $\bar{K}$  its algebraic closure and
$\Gal(K)=\Aut(\bar{K}/K)$ its absolute Galois group. 
Let us assume that $K$ contains 
primitive $p$th root of unity and denote it by $\gamma$.

\begin{sect} {\bf Galois groups of odd polynomials}.
Let $r \ge 2$ be an  integer and $n=2pr-1$.  Let us put
$$m=\frac{n-1}{2}=pr-1,$$
Clearly, $m \ge 5$; it   is odd if and only if $r$ is even.
Let
$f(x)\in K[x]$ be a degree $n$ {\sl odd} polynomial {\sl without repeated roots}.
Then  $0$ is simple root of $f(x)$ and
 there exist
$m$ distinct non-zero roots $\{\beta_1, \dots , \beta_m\}$ of $f(x)$ such that
the $n$-element set $\RR_f$ of roots of $f(x)$ coincides with $\{0\} \cup \{\pm
\beta_1, \dots , \pm \beta_m \}\subset \bar{K}$ of all roots of $f(x)$.
Clearly, $\RR_f$ is Galois-stable. We write $\Perm(\RR_f)$ for the group of
permutations of the $n$-element set $\RR_f$. Let $\Gal(f)$ be the image of
$\Gal(K)$ in $\Perm(\RR_f)$. If $K(\RR_f)\subset \bar{K}$ is the splitting field of $f(x)$ obtained by
adjoining to $K$ all elements of $\RR_f$ then $K(\RR_f)/K$ is a finite Galois
extension and  $\Gal(f)$ is canonically isomorphic to the Galois group
$\Gal(K(\RR_f)/K)$. Let $\Perm_0(\RR_f)$ be the subgroup of $\Perm(\RR_f)$ that
consists of all permutations of the form
 $$(\epsilon, s): 0\mapsto 0, \ \beta_i \mapsto \epsilon(i) \beta_{s(i)}, \ -\beta_i \mapsto -\epsilon(i) \beta_{s(i)}$$
 where $s \in \ST_m$ is an arbitrary permutation of $\{1,2, \dots, m\}$ and
  $\epsilon: \{1,2, \dots, m\} \to \{\pm 1\}$ is an arbitrary function  on $\{1,2, \dots, m\}$ that takes only values $1$ and $-1$.
Clearly,
$$\Gal(f)\subset \Perm_0(\RR_f) \subset \Perm(\RR_f).$$
We write $E_m$ for the subgroup of $\Perm_0(\RR_f) $ that consists of all permutations of the form
$$(\epsilon, \mathrm{id}): 0 \mapsto 0, \ \beta_i \mapsto \epsilon(i) \beta_i, \ -\beta_i \mapsto -\epsilon(i) \beta_i$$
where $\mathrm{id}$ is the identity permutation of $ \{1,2, \dots,m\}$.
Clearly, $E_m\cong (\Z/2\Z)^m$.
Let $\W(\D_m)$ be the index $2$ subgroup of $\Perm_0(\RR_f)$, whose elements are characterized by
the condition
$\prod_{i=1}^m \epsilon(i) =1$. We have
$$\W(\D_m)\subset \Perm_0(\RR_f) \subset \Perm(\RR_f), \quad  \W(\D_m) \cap E_m=E_m^0$$
where $E_m^0$ is an index $2$ subgroup of $E_m$ that consists of all the permutations of the form
$$(\epsilon, \mathrm{id}): 0 \mapsto 0, \ \beta_i \mapsto \epsilon(i) \beta_i, \ -\beta_i \mapsto -\epsilon_i \beta_i$$
with $\prod_{i=1}^m \epsilon(i) =1$.

Notice that $\W(\D_m)$ is the semi-direct product of its normal group
$E_m^0\cong (\Z/2\Z)^{m-1}$ and the subgroup $\tilde{\ST}_m$. Here
 and
$\tilde{\ST}_m$ is identified with the subgroup of all permutation of the form
$$(\mathbf{1},s): 0 \mapsto 0, \ \beta_i \mapsto  \beta_{s(i)}, \ -\beta_i \mapsto -\beta_{s(i)}.$$
 and 
$\mathbf{1}$ is the constant function $1$ on  $ \{1,2, \dots,m\}$. 
Direct computations show that
\begin{equation}
\label{adjoint}
(\mathbf{1},s) (\epsilon, \mathrm{id}) (\mathbf{1},s)^{-1}= (\epsilon \circ s^{-1}, \mathrm{id}) 
\end{equation}
for all $s \in \ST_m, \ (\epsilon, \mathrm{id}) \in E_m$. Here $\epsilon \circ s^{-1}$ is the composition
$$i \mapsto \epsilon(s^{-1}i)\in \{\pm 1\} \quad \forall i \in \{1,2, \dots,m\}.$$
This provides $E_m$ with the natural structure of a $\ST_m$-module. If $\mathfrak{G}$ is a subgroup of $\ST_m$
then we may view $E_m$ as the $\mathfrak{G}$-module as well.

Since $0$ is a simple root of $f(x)$, we have $f(x)=x\cdot  h (x)$ where $h(x)$
is an {\sl even} polynomial without repeated roots of  degree $2m=n-1$, whose set of roots $\RR_{h}$
is $\{\pm \beta_1, \dots , \pm \beta_m\}$
where
$$\beta_i \ne 0, \quad \beta_i \ne \pm \beta_j \ \ \forall i \ne j.$$
In particular, $h(0)\ne 0$. Since $h(x)$ is even and without repeated roots,
there is a degree $m$ polynomial $u(x) \in K[x]$ without repeated roots such that
\begin{equation}
\label{u2h}
h(x)=u(x^2), \ \ h(0)=u(0);
\end{equation}
in addition,  the $m$-element set  $\RR_u$  of roots of $u(x)$ coincides with
$$\{\alpha_1=(\pm \beta_1)^2, \dots, \alpha_i=(\pm \beta_i)^2, \dots,  (\pm \beta_m)^2\}.$$
 This implies that 
$$K(\RR_u)\subset K(\RR_h)=K(\RR_f),$$
and $K(\RR_f)/K(\RR_u)$ is a finite Galois extension.
Each $(\epsilon,s) \in \Gal(f)$ acts on $\RR(u)\subset K(\RR_u)$ via
$$\alpha_i=\beta_i^2 \mapsto (\pm \beta_{s(i)})^2=\beta_{s(i)}^2=\alpha_{s(i)},$$
i.e.,
$$(\epsilon,s)(\alpha_i)=\alpha_{s(i)}.$$
This gives an explicit description of 
the natural surjective homomorphism
$$\kappa_u: \Gal(f)=\Gal(h) \twoheadrightarrow \Gal(u)\subset \Perm(\RR_u)=\ST_m$$
induced by the field inclusion $K(\RR_u)\subset K(\RR_h)$. Namely,
\begin{equation}
\label{kappaGal}
\kappa_u((\epsilon,s))=s \quad \forall (\epsilon,s) \in \Gal(h)\subset \Perm_0(\RR_h).
\end{equation}
\end{sect}
\begin{rem}
\label{kerKappa}
Clearly, $\ker(\kappa_u)$ is a normal subgroup of $\Gal(f)$. It follows from  \eqref{kappaGal} that
$\ker(\kappa_u)$ is a subgroup of $E_m$. On the other hand,  i \eqref{adjoint}
combineded with the normality implies that $\ker(\kappa_u)$ is a $\Gal(u)$-submodule of $E_m$.
It follows from the very definition of surjective $\kappa_u$ that its kernel is trivial if and only if 
$K(\RR_u)= K(\RR_h)$. In general, by Galois theory, $\ker(\kappa_u)$ is canonically isomorphic 
to the Galois group $\Gal( K(\RR_h)/ K(\RR_u))$.
\end{rem}

\begin{lem}
\label{evenR}
\begin{itemize}
\item[(i)]
If $r$ is even then $\Gal(f)=\Gal(h)\subset \W(\D_m)$ if and only if $-u(0)$
is a square in $K$.
\item[(ii)]
Suppose that $r$ is even, $-u(0)$
is a square in $K$. and $\Gal(u)$ is either $\ST_m$ or the alternating group $\mathbf{A}_m$.
If $K(\RR_h) \ne K(\RR_u)$ then $\ker(\kappa_u)=E_m^0$. In particular, if 
$\Gal(u)=\ST_m$ then
$$\Gal(f)=\Gal(h)=\W(\D_m).$$
\end{itemize}
\end{lem}

\begin{proof}
(i) follows readily from Remark 2.2 in \cite{ZarhinI}, since
 $u(0)=h(0)$ by \eqref{u2h}.

Let us prove (ii).  Since $r$ is even, $m$ is odd.
Thanks to (i),
$$\Gal(f)=\Gal(h)\subset\W(\D_m).$$
This implies that $\ker(\kappa_u)\subset E_m^0$. Since
$K(\RR_h) \ne K(\RR_u)$,  $\ker(\kappa_u)\ne \{0\}$.
So, $\ker(\kappa_u)$ is a nonzero $\mathbf{A}_m$-submodule of $E_m^0$.
However, one easily check that $\mathbf{A}_m$-module $E_m^0$  is nothing else
but the heart of the standard permutation representation of $\mathbf{A}_m$ over the
two-element field $\mathbb{F}_2$ \cite{Mortimer}. But this representation is irreducible (ibid).
This implies that $\ker(\kappa_u)= E_m^0$, which ends the proof.
\end{proof}

\begin{defn}
Let $\Perm(\RR_h)$ be the group of permutations of the $2m$-element set
$\RR_h$.
 Let $\GG$ be a permutation subgroup in $\ST_m$. We write $2^m \cdot \GG\subset \Perm(\RR_h)$ for
 the subgroup of all permutations of
the form
$$(s, \epsilon):
\beta_i \mapsto \epsilon(i) \beta_{s(i)}, \ -\beta_i \mapsto -\epsilon(i)
\beta_{s(i)}$$
 where $$s \in \GG, \quad \epsilon: \{1,2, \dots, m\} \to \{\pm 1 \}.$$ We write $2^{m-1} \cdot \GG$ for
 the index two subgroup in $2^m \cdot \GG$, whose elements $(s, \epsilon)$ are characterized by the
 condition $\prod_{i=1}^m \epsilon(i) =1$.
\end{defn}

\begin{ex}[Example 2.5 of \cite{ZarhinI}]
\item[(i)] The group $2^{m-1} \cdot \{1\}$ coincides with the group $E_m^0$ of all
permutations of the form
$$\beta_i \mapsto \epsilon{i} \beta_{i}, \ -\beta_i \mapsto -\epsilon{i}
\beta_{i}$$  where $ \epsilon_i=\pm 1$ while $2^{m-1} \cdot \{1\}$ corresponds
to its index $2$ subgroup, whose elements are characterized by the condition
 $\prod_{i=1}^m \epsilon_i =1$. 
 This is a subgroup of index 2 in $E_m=2^m \cdot \{1\}$.The groups $2^m \cdot \{1\}$ and  $2^{m-1} \cdot
 \{1\}$ are exponent $2$ commutative groups of order $2^m$ and $2^{m-1}$
 respectively.
 \item[(ii)]
 Let us identify $\Perm(\RR_h)$ with the stabilizer of $0$ in $\Perm(\RR_f)$. Then
$2^m\cdot \ST_m$ coincides with $\Perm_0(\RR_f)$ and $2^{m-1}\cdot \ST_m$
coincides with $\W(\D_m)$.
\end{ex}


\begin{rems}[Remark 2.7 of \cite{ZarhinI}]
\label{orbits} Suppose that there exists a  permutation group $\GG \subset
\ST_m$ such that $\Gal(h) =2^{m-1} \cdot \GG$.
 Then:

\begin{itemize}
\item[(i)] 
 If $\GG$ does not contain a normal subgroup of odd index (except
 $\GG$ itself) then $\Gal(h)$ also does not contain a normal subgroup of odd index (except
 $\Gal(h)$ itself).
\item[(ii)]
\begin{enumerate}
\item If $\GG$ is a transitive subgroup of $\ST_m$ then 
$\Gal(h)$ is a
transitive subgroup of $\Perm(\RR_h)$, i.e., $h(x)$ is irreducible over $K$.

 \item
 Suppose that $\GG$ is a doubly transitive subgroup of $\ST_m$ and let
$\GG_1$ is the stabilizer of $1$ in $\GG$. 
Then $\GG_1$ has exactly two orbits
in $\{1, \dots , m\}$: namely, $\{1\}$ and the rest.
 Let $\Gal(h)_1$ be the stabilizer of
$\beta_1$ in $\Gal(h)$. Then one may easily check that $\Gal(h)_1$ has exactly $3$ orbits in
$\RR_h$: namely, $\{\beta_1\}$, $\{-\beta_1\}$ and the rest.
\end{enumerate}
\end{itemize}
\end{rems}

\begin{sect}
{\bf Permutation modules}. Let $V_{{f}}$ be the
 $2m$-dimensional $\F_p$-vector space of functions $$\phi:{\RR}_{f}\to
 \F_p, \ \sum_{\alpha\in {\RR}_{{f}}}\phi(\alpha)=0.$$
 The space $V_{{f}}$ carries the natural structure of Galois module
 induced by the Galois action on ${\RR}_{{f}}$.

 Let $\F_p^{\RR_h}$ be the
$2m$-dimensional $\F_p$-vector space of all functions $\phi:{\RR}_{h}\to
 \F_p$.
  It carries the natural structure of a Galois module. We write $1_{\RR_h}$ for the
  (Galois-invariant) constant function $1$.

  The map that assigns
 to a $\F_p$-valued function on $\RR_f$ its restriction to $\RR_h$ gives rise to
 the isomorphism $V_f \to \F_p^{\RR_h}$ of Galois modules. (One may extend a
 function $\phi$ on $\RR_h$ to $\RR_f=\{0\}\cup \RR_h$ by putting
 $$\phi(0):=-\sum_{\alpha\in \RR_h}\phi(\alpha).)$$
 The Galois module $V_{{f}}$ splits into a
direct sum of the Galois submodules of odd and even functions
$$V_{{f}}=V_{{f}}^{-}\oplus V_{{f}}^{+}$$
where
$$V_{{f}}^{+}=\{\phi:{\RR}_{f}\to \F_p, \sum_{\alpha \in \RR_f} \phi(\alpha)=0, \ \phi(\alpha)= \phi(-\alpha) \ \forall \alpha\},$$
$$V_{{f}}^{-}=\{\phi:{\RR}_{f}\to \F_p, \ \phi(\alpha)= -\phi(-\alpha) \ \forall \alpha\}.$$
(The sum of values of an odd function is always zero.) Clearly,
 $\phi(0)=0$ for all $\phi \in V_{{f}}^{-}$. It follows that
$$\dim_{\F_p}(V_{{f}}^{-})=m.$$
\end{sect}

\begin{lem}
\label{centralizer} Suppose that  there exists a doubly transitive permutation
group $\GG \subset \ST_m$ such that $\Gal(h)=2^{m-1}\cdot\GG$.
 Then $\End_{\Gal(K)}(V_{{f}}^{-})=\F_p$.
\end{lem}

\begin{proof}
By Remark \ref{orbits}(ii)(1), $\Gal(h)$ acts transitively on $\RR_h$.
 Let $W_h^{+}$ and $W_h^{-}$ be the subspaces of even and odd
functions respectively in $\F_p^{\RR_h}$. Clearly, they both are Galois
submodules of $\F_p^{\RR_h}$, and 
$$W_h^{-}\oplus W_h^{+} =\F_p^{\RR_h}.$$ It
is also clear that the Galois modules $W_h^{-}$ and $V_{{f}}^{-}$ are
isomorphic. So, it suffices to check that
$$\End_{\Gal(K)}(W_{{h}}^{-})=\F_p.$$
In order to do that, notice that $\#(\RR_h)=2m=n-1=2pr-2$ is {\sl not}
divisible by $p$. This implies that the submodule $\F_p \cdot 1_{\RR_h}$ of constant functions is a
direct summand of $W_h^{+}$ and therefore $\F_p^{\RR_h}$ splits into a direct sum of
Galois modules
$$\F_p^{\RR_h}=W_h^{-}\oplus W_h^{+}=W_h^{-}\oplus \F_p \cdot 1_{\RR_h}
\oplus W_h^{+,0}$$ where $W_h^{+,0}$ is the Galois (sub)module of even
functions, whose sum of values is zero. Clearly,
$$\dim_{\F_p}\End_{\Gal(K)} (\F_p^{\RR_h})\ge $$
$$\dim_{\F_p}\End_{\Gal(K)}(W_h^{-}) + \dim_{\F_p}\End_{\Gal(K)}(\F_p \cdot
1_{\RR_h}) + \dim_{\F_p}\End_{\Gal(K)}(W_h^{+,0}) \ge $$
$$\dim_{\F_p}\End_{\Gal(K)}(W_h^{-}) +1 +1.$$ So, if we prove that
$\dim_{\F_p}\End_{\Gal(K)} (\F_p^{\RR_h})=3$, then we are done. Since the image
of $\Gal(K)$ in $\Aut_{\F_p}(\F_p^{\RR_h})$ coincides with
$$\Gal(h)\subset \Perm(\RR_h)\subset \Aut_{\F_p}(\F_p^{\RR_h}),$$
we have
$$\End_{\Gal(K)} (\F_p^{\RR_h})=\End_{\Gal(h)} (\F_p^{\RR_h}).$$
So, in order to prove the Lemma, it suffices to check that
$$\dim_{\F_p}(\End_{\Gal(h)} (\F_p^{\RR_h}))=3.$$
By Lemma 7.1 of \cite{Passman}, $\dim_{\F_p}(\End_{\Gal(h)} (\F_p^{\RR_h}))$
coincides with the number of orbits in $\RR_h$ of the stabilizer  in $\Gal(h)$
of any root of $h(x)$. But the number of orbits is $3$ (see Remark
\ref{orbits}(ii)(2)). This ends the proof.
\end{proof}

\section{Cyclic covers, jacobians and prymians}
\label{geom}

If $X$ is an abelian variety over $\bar{K}$ then we write $\End(X)$ for the
ring of its $\bar{K}$-endomorphisms and $\End^0(X)$ for the
corresponding $\Q$-algebra $\End(X)\otimes\Q$. If $X$ is defined over $K$ then
we write $\End_K(X)$ for the ring of its $K$-endomorphisms.

As above $f(x)=x\cdot h(x)\in K[x]$ is an odd polynomial of degree $n=2m+1=2pr-1$
without repeated roots. We keep all the notation of the previous Section.

\begin{sect}
{\bf Superelliptic curves}. Hereafter we assume that $K$  contains $\zeta_p$. Let
us consider the smooth projective model $C_{f,p}$ of the plane affine
 curve
$y^p=f(x)$; the genus of $C_{f,p}$ is $(n-1)(p-1)/2=m(p-1)$.
 The curve $C_{f,p}$ admits commuting periodic
automorphisms
$$\delta_2:(x,y)\mapsto (-x,-y)$$
and
$$\delta_p:(x,y)\mapsto (x,\zeta_p y)$$
of period $2$ and $p$ respectively.

The regular map of curves
$$\pi:C_{f,p}\to \P^1, \ (x,y)\mapsto x$$
has degree $p$ and branches exactly at $0$, the $2m$-element set 
$\{\alpha\mid \alpha\in \RR_f\}$ and $\infty$. (Notice that $p$ does {\sl not}
divide $2m+1=n$.) Clearly, all  branch points of $\pi$ in $C_{f,p}$ are
$\delta_p$-invariant. By abuse of notation, we denote $\pi^{-1}(\infty)$ by
$\infty$. Let us put
$$B=\pi^{-1}(\RR_f)=\{(\alpha,0)\mid \alpha\in \RR_f\}\subset C_{f,p}(\bar{K}).$$
Clearly, all elements of $B$ are $\delta_p$-invariant. On the other hand, if
$P= (\alpha,0)\in B$ then $\delta_2(P)=(-\alpha,0)\in B$.

The automorphism $\delta_2:C_{f,p} \to C_{f,p}$ has exactly two fixed points,
namely, $\pi^{-1}(0)$ and $\pi^{-1}(\infty)$. One may easily check that the
quotient $\tilde{C}_{f,p}=C_{f,p}/(1,\delta_2)$ is a smooth (irreducible)
projective curve (compare with Lemma 1.2, its proof and Corollary 1.3 in
\cite{ZarhinMF}) and $C_{f,p}\to \tilde{C}_{f,p}$ is a double covering that is
ramified at exactly two points, namely $\pi^{-1}(0)$ and
$\pi^{-1}(\infty)$. The Hurwitz formula implies that the genus of
$\tilde{C}_{f,p}$ is $m(p-1)/2$.

It follows from (\cite{ZarhinCrelle}, \cite[Remarks 3.5 and 3.7]{ZarhinCamb}) that the 
$(n-1)(p-1)$-dimensional
$\bar{K}$-vector space $\Omega^1(C_{f,p})$ of differentials of the first kind
on $C_{f,p}$ has a basis
$$\left\{ x^i \frac{dx}{y^j}, \ \
1 \le j \le p-1, \quad 0 \le i \le [nj/p]-1\right\}.$$ 
If
$$\delta_2^{*}:\Omega^1(C_{f,p})\to \Omega^1(C_{f,p}), \ \delta_p^{*}:\Omega^1(C_{f,p})\to
\Omega^1(C_{f,p})$$ are the automorphisms induced by $\delta_2$ and $\delta_p$
respectively then
$$\delta_p^{*}(x^i \frac{dx}{y^j})=\zeta_p^{-j}x^i \frac{dx}{y^j},$$

$$\delta_2^{*}(x^i \frac{dx}{y^j})=(-1)^{i+1-j} x^i \frac{dx}{y^j}.$$ 
In particular, the
 basis consists of eigenvectors with respect to both $\delta_2^{*}$ and $\delta_p^{*}$.
It follows that the subspace $\Omega^1(C_{f,p})^{-}$ of
$\delta_2$-anti-invariants is $m$-dimensional and admits a basis
$$\left\{ x^{i} \frac{dx}{y^j}, \ 1 \le j \le p-1, \ 0\le i \le [nj/p]-1; \quad  i+1+j \ \ \text{is odd} \right\}.$$
Taking into account that $n=2pr-1$, we get
$$[nj/p]-1=[(2prj-j)/p]-1=(2rj-1)-1=2rj-2$$
(recall that $1 \le j \le p-1$). This means that the basis of $\Omega^1(C_{f,p})^{-}$ consists of forms
$$\left\{ x^{i} \frac{dx}{y^j}, \ 1 \le j \le p-1, \ 0\le i \le 2rj-2; \quad  i, j \ \ \text{have the same parity} \right\}.$$
It follows that our basis partitions in $(p-1)$ parts $B_j$ indexed by $j$
where for {\sl even} $j$
$$B_j=\left\{ x^{i} \frac{dx}{y^j}, \ \
i=0,2, \dots, 2rj-2\right\}$$
is a $rj$-element set while for {\sl odd} $j$
$$B_j=\left\{ x^{i} \frac{dx}{y^j}, \ \
i=1,3, \dots, 2rj-3\right\}$$
is a $(rj-1)$-element set. In particular, the cardinalities of $B_j$ are distinct for different $j$ (recall that $r \ge 2$).

\end{sect}

\begin{sect}
{\bf Superelliptic jacobians}. Let $J(C_{f,p})$ be the jacobian of $C_{f,p}$: it is
a $m(p-1)$-dimensional abelian variety that is defined over $K$. By Albanese
functoriality, $\delta_2$ and $\delta_p$ induce the $K$-automorphisms of
 $J(C_{f,p})$ that we still denote by $\delta_2$ and $\delta_p$ respectively.
 We have
 $$\delta_2^2=1, \ \sum_{i=0}^{p-1} \delta_p^i=0$$
 where all the equalities hold in $\End(J(C_{f,p}))$.
 The latter equality gives rise to the embedding
 $$\Z[\zeta_p] \hookrightarrow \End_K(J(C_{f,p}))\subset \End(J(C_{f,p})), \ \zeta_p \mapsto
 \delta_p$$
 of the cyclotomic ring $\Z[\zeta_p]$ into the  ring of $K$-endomorphisms of $J(C_{f,p})$.

 Let $\alb:C_{f,p} \hookrightarrow J(C_{f,p})$ be the canonical embedding of $C_{f,p}$ into
 its jacobian normalized by the condition $\alb(\infty)=0$, i.e., $\alb$ sends a point
 $Q \in C_{f,p}(\bar{K})$ to the linear equivalence class of the divisor
 $(Q)-(\infty)$.
 
  Recall that $\infty\in C_{f,p}$ is $\delta_2$-invariant and $\delta_p$-invariant; it goes to the  zero $O$ of $J(C_{f,p})$  under $\alb$;
  in particular,
  $$\delta_2(O)=O, \quad \delta_p(O)=O.$$
  Of course, $O$ is invariant under any endomorphism of the abelian variety $J(C_{f,p})$.
 In light of the universality property of the Albanese map \cite[Ch. II, Sect. 2, Th.9]{Lang}, 
 there exist $Q_2, Q_p \in J(C_{f,p})(\bar{K})$ such that for all $P \in J(C_{f,p})(\bar{K})$
 $$\alb \ \delta_2(P)=\delta_2 \ \alb(P)+Q_2, \quad \alb \ \delta_p(P)=\delta_p\ \alb(P)+Q_p.$$
 If we put $P=\infty$ then we get
 $$O=Q_2, \quad O=Q_p.$$
 This means that 
 $$\alb \ \delta_2=\delta_2 \ \alb, \quad \alb \ \delta_p=\delta_p\ \alb,$$
 i.e., $\alb$ is $\delta_p$-equivariant and
 $\delta_2$-equivariant.

 Let us remind the description of the Galois (sub)module $J(C_{f,p})^{\delta_p}$ of
 $\delta_p$-invariants in $J(C_{f,p})(\bar{K})$.  The Galois modules
$V_{{f}}$ and $J(C_{f,p})^{\delta_p}$ are canonically isomorphic
\cite{Schaefer} (see also \cite{ZarhinMiami}). Namely, let
$$\Z_B^0=\{\sum_{P\in B}a_P(P)\mid a_P\in \Z \ \forall P \in B, \ \sum_{P\in B}a_P=0 \}$$
be the group of degree zero divisors on $C_{f,p}$ with support in $B$. The free
commutative group $\Z_B^0$ carries the natural structure of a Galois module.
Clearly, the Galois module $\Z_B^0/p \Z_B^0$ is canonically isomorphic to
$V_{{f}}$: a divisor $\sum_{P\in B}a_P(P)$ gives rise tho the function $\alpha
\mapsto a_P \mod p$ where $P=(\alpha,0)$. Since $B$ is $\delta_2$-invariant,
the map
$$D_2: \sum_{P\in B}a_P(P)\mapsto \sum_{P\in B}a_P \delta_2(P)$$
is an automorphism of  the Galois module $\Z_B^0$ that induces the
automorphism of $\Z_B^0/p \Z_B^0=V_{{f}}$ that sends a function $\alpha \to
\phi(\alpha)$ to the function $\alpha \to \phi(-\alpha)$. (We still denote this
automorphism of $V_{{f}}$ by $D_2$.)  Notice that
$$V_{{f}}^{-}=(1-D_2)V_{{f}}, \ V_{{f}}^{+}=(1+D_2)V_{{f}}$$
(recall that $V_f$ is a $\F_p$-vector space and $p$ is odd.) In other words, $V_{{f}}^{+}$
and $V_{{f}}^{-}$ are eigenspaces of $D_2$ that correspond to eigenvalues $1$
and $-1$ respectively.

Let us consider the natural map
$$\cl: \Z_B^0 \to J(C_{f,p})(\bar{K})$$
that sends a divisor $\sum_{P\in B}a_P(P)$ to (its linear equivalence class,
i.e., to)
 $$\sum_{P\in B}a_P \alb(P) \in J(C_{f,p})(\bar{K}).$$ 
 It turns out that
$\cl(\Z_B^0)=J(C_{f,p})^{\delta_p}$ and the kernel of $\cl$ coincides with
$p\cdot \Z_B^0$. (If $b_1, b_2 \in B$ then the  degree $p$ divisors $p(b_1)$ and $p(b_2)$
are linear equivalent, since the both are pulled back from $\mathbb{P}^1$ via $\pi$).
This gives rise to the natural isomorphism of Galois module
$\Z_B^0/p \Z_B^0$ and $J(C_{f,p})^{\delta_p}$ and we get the natural
isomorphisms of Galois modules
$$\overline{\cl}: V_{{f}}=\Z_B^0/p \Z_B^0\cong J(C_{f,p})^{\delta_p}.$$
Since $\delta_2$ commutes with $\delta_p$, the Galois submodule
$J(C_{f,p})^{\delta_p}$ is $\delta_2$-invariant. It follows from the explicit
description of $\cl$ and $D_2$ that if $\bar{\cl}(\phi)=P \in
J(C_{f,p})^{\delta_p}$ then $\delta_2 (P)$ is the image (under $\overline{\cl}$)
of the function $\alpha \to \phi(-\alpha)$. In other words,
$$\overline{\cl}(D_2 \phi)=\delta_2 \left(\overline{\cl}(\phi) \right)\ \forall \phi \in
V_f.$$ It follows that the restriction of $\overline{\cl}$ to $V_{{f}}^{-}$
gives us the isomorphism of Galois modules
$$\overline{\cl}: V_{{f}}^{-} \cong \{P \in J(C_{f,p})^{\delta_p}\mid
\delta_2 (P)=-P\}.$$ This implies that
$$\{P \in J(C_{f,p})^{\delta_p}\mid
\delta_2(P)=-P\}=\overline{\cl}(V_{f}^{-})=\overline{\cl}((1-D_2)V_{f})
 =(1-\delta_2)J(C_{f,p})^{\delta_p}.$$
\end{sect}

\begin{sect}
\label{Triprym} {\bf Superelliptic prymians}.
 Let us consider the Prym variety
 $$\mathrm{Prym}(C_{f,p})=(1-\delta_2)J(C_{f,p})\subset J(C_{f,p}).$$
 If one restricts the canonical principal polarization on $J(C_{f,p})$ to
$\mathrm{Prym}(C_{f,p})$ then the induced polarization is twice a principal polarization on
$\mathrm{Prym}(C_{f,p})$ \cite[Sect. 3, Cor. 2]{MumfordP}.
Obviously, the principal polarization on
$\mathrm{Prym}(C_{f,p})$ is $\delta_p$-invariant.
It is also clear \cite[Sect. 3,
Cor. 2]{MumfordP} that $\mathrm{Prym}(C_{f,p})$ coincides with the identity component of
the surjective map of jacobians $J(C_{f,p})\to J(\tilde{C}_{f,p})$; in
particular, it is a $m(p-1)/2$-dimensional abelian variety that is defined over $K$.
Clearly, the abelian subvariety $\mathrm{Prym}(C_{f,p})$ is $\delta_p$-invariant. Therefore
we may and will consider the restriction of $\delta_p$ to $\mathrm{Prym}(C_{f,p})$ as the $K$-automorphism $\delta_{p,-}$ of $\mathrm{Prym}(C_{f,p})$. Still
$\sum_{i=0}^{p-1}\delta_{p,-}^i=0$
 in $\End(\mathrm{Prym}(C_{f,p}))$. As above, this induces an
embedding $$\Z[\zeta_p] \hookrightarrow \End(\mathrm{Prym}(C_{f,p})), \ \zeta_p \mapsto
 \delta_{p,-}.$$
On the other hand, $1+\delta_2$ kills $\mathrm{Prym}(C_{f,p})$, because
$$0=1-\delta_2^2=(1+\delta_2)(1-\delta_2)\in \End(J(C_{f,p}))$$ and
$\mathrm{Prym}(C_{f,p})(\bar{K})=(1-\delta_2)(J(C_{f,p}))$. This implies that
$$\delta_2 (P) = -P \ \forall P \in \mathrm{Prym}(C_{f,p})(\bar{K}).$$
Let us consider the Galois (sub)module $\mathrm{Prym}(C_{f,p})^{\delta_{p,-}}$ of
$\delta_{p,-}$-invariants in  $\mathrm{Prym}(C_{f,p})(\bar{K})$. Clearly,
$$\mathrm{Prym}(C_{f,p})^{\delta_{p,-}}\subset \{P \in J(C_{f,p})^{\delta_p}\mid
\delta_2(P)=-P\}.$$ Since the latter group coincides with
$(1-\delta_2)J(C_{f,p})^{\delta_p}$, we conclude that
$$\mathrm{Prym}(C_{f,p})^{\delta_{p,-}}= \{P \in J(C_{f,p})^{\delta_p}\mid
\delta_2(P)=-P\}.$$
It follows that the Galois modules $\mathrm{Prym}(C_{f,p})^{\delta_{p,-}}$
and $V_f^{-}$ are canonically isomorphic.

Let us put $$\Oc=\Z[\zeta_p], \ \lambda=(1-\zeta_p)\Oc, \ \
E=\Oc\otimes\Q=\Q(\zeta_p).$$ Then the residue field
$$\Oc/\lambda=\F_p.$$  Recall that we have the natural homomorphism
$$\Oc=\Z[\zeta_p]\hookrightarrow \End_K(\mathrm{Prym}(C_{f,p})), \ \zeta_p\mapsto \delta_p.$$
This implies that
$$\mathrm{Prym}(C_{f,p})^{\delta_{p,-}}=\mathrm{Prym}(C_{f,p})_{\lambda}$$
and therefore the Galois modules  $\mathrm{Prym}(C_{f,p})_{\lambda}$ and $V_{f}^{-}$ are
canonically isomorphic. In particular,
$$\dim_{\F_p}(\mathrm{Prym}(C_{f,p})_{\lambda})=m.$$
On the other hand, it is well known \cite{Shimura,Ribet2,ZarhinMZ} that $\mathrm{Prym}(C_{f,p})_{\lambda}$ is a free
$\Oc/\lambda$-module of rank $2\dim(\mathrm{Prym}(C_{f,p}))/[E:\Q]$. Since
$\Oc/\lambda=\F_p$ and $[E:\Q]=p-1$, we get another proof of the equality $\dim(\mathrm{Prym}(C_{f,p}))=m(p-1)/2$.
Notice that
 $$\dim(\mathrm{Prym}(C_{f,p}))=m(p-1)/2=\dim_{\C}(\Omega^1(C_{f,p})^{-}).$$
\end{sect}

\begin{rem}
\label{multprime}
 Taking into account that $\dim(\mathrm{Prym}(C_{f,p}))=m(p-1)/2$ and applying Theorem 3.10 of \cite{ZarhinMiami} to
$$ Y=J(C_{f,p}), \  \ Z=\mathrm{Prym}(C_{f,p}), \ \ \delta=\delta_2, \ \mathcal{P}(t)=1-t,$$
we obtain that $$1-\delta_2:J(C_{f,p})\twoheadrightarrow \mathrm{Prym}(C_{f,p})\subset
J(C_{f,p})$$ induces (on differentials of the first kind) an isomorphism
$$(1-\delta_2)^{*}:\Omega^1(\mathrm{Prym}(C_{f,p}))\cong \Omega^1(C_{f,p})^{-}\subset
\Omega^1(J(C_{f,p}))$$
 and this isomorphism is $\delta_p$-equivariant. It
follows easily from considerations of Subsection \ref{Triprym} that
the spectrum of the linear operator 
$$\delta_{p,-}^{*}:
\Omega^1(\mathrm{Prym}(C_{f,p})) \to \Omega^1(\mathrm{Prym}(C_{f,p}))$$
coincides with the set $\{\zeta_p^{-j} \mid 1 \le j \le p-1\}$ of primitive $p$th roots of unity; in addition the multiplicity $\mathrm{mult}(\zeta_p^{-j})$
of eigenvalue $\zeta_p^{-j}$ is $rj$ if $j$ is even and $rj-1$ if $j$ is odd.
It follows that GCD of all these multiplicities is $1$ if and only if $r$ is even. Notice also that all these multiplicities are distinct if $r$ is even.

\end{rem}

\begin{thm}
\label{main}
 Assume that $r$ is even and there exists a doubly transitive permutation group $\GG \subset
 \ST_m$ that enjoys the following properties:

 \begin{itemize}
 \item[(i)]
 $\GG$ does not contain a normal subgroup, whose index divides $m$ (except $\GG$
 itself).
\item[(ii)] $\Gal(h)$ contains $2^{m-1}\cdot \GG$.
 \end{itemize}	
 Then $\End(\mathrm{Prym}(C_{f,p}))=\Z[\delta_{p,-}] \cong\Z[\zeta_p]$.
 In particular, $\mathrm{Prym}(C_{f,p})$ is an absolutely simple abelian variety.
\end{thm}

\begin{proof}
Enlarging $K$ if necessary, we may and will assume that $\Gal(h)= 2^{m-1}\cdot
\GG$. Identifying $\Perm(\RR_h)$ with the stabilizer of $0$ in $\Perm(\RR_f)$,
we obtain that
$$\Gal(f)=\Gal(h)=2^{m-1}\cdot\GG.$$
Since the Galois modules  $\mathrm{Prym}(C_{f,p})_{\lambda}$ and $V_f^{-}$ are isomorphic,
it follows from Lemma \ref{centralizer} that
$\End_{\Gal(K)}(\mathrm{Prym}(C_{f,p})_{\lambda})=\F_p$.
 Applying Remark \ref{multprime} and 
 Theorem 6.11 of \cite{ZarhinKumar}
$$X=\mathrm{Prym}(C_{f,p}), \ E=\Q(\zeta_p), \
\Oc=\Z[\zeta_p], \ \lambda=(1-\zeta_p)\Oc,$$
we obtain that $\Z[\delta_{p,-}]$ coincides with its own centralizer in 
$\End(\mathrm{Prym}(C_{f,p}))$ and the $\Q$-subalgebra $\Q[\delta_{p,-}]$ of
$\End^0(\mathrm{Prym}(C_{f,p}))$ coincides with its own centralizer. In particular, $\Q[\delta_{p,-}]$ contains the center $\mathcal{Z}$ of 
$\End^0(\mathrm{Prym}(C_{f,p}))$. On the other hand, it follows from Theorem 2.3 of \cite{ZarhinCamb} applied to $Z=\mathrm{Prym}(C_{f,p})$ and $E=
\Q[\delta_{p,-}]$ that  if $\Q[\delta_{p,-}]$ does {\sl not}
coincides with the center then all the multiplicities $\mathrm{mult}(\zeta_p^{-j})$ could {\sl not} be distinct. However, they all are distinct! This implies that
$\Q[\delta_{p,-}]$ coincides with the center and therefore also coincides with the whole $\End^0(\mathrm{Prym}(C_{f,p}))$. It follows that
$\Z[\delta_{p,-}]=\End(\mathrm{Prym}(C_{f,p}))$, which ends the proof.
\end{proof}

\begin{ex}	
Let $\bar{\Q}$ be the algebraic closure of $\Q$ in $\C$.
Let $p$ be an odd prime, $r \ge 2$ be an integer such that $(r-1)$ and $(p-1)$ are relatively prime. Let us put
$$m:=pr-1, \ n=2m+1=2pr-1 \ge 5.$$
Clearly, $m \ge 3 \cdot 2-1=5$.

  Let $L$ be the field of rational functions $\bar{\Q}(t_1, \dots , t_m)$ in $m$ independent
 variables $t_1, \dots , t_m$ over  $\bar{\Q}$.
 One may realize $2^{m}\cdot\ST_m$ as the following  group of (linear) automorphisms of $L$:
 $$(s; \epsilon_1, \dots , \epsilon_m):
t_i \mapsto \epsilon_i t_{s(i)}, \ i=1, \dots m$$
 where $$s \in \ST_m, \ \epsilon_i=\pm 1 .$$
 Let $K$ be the subfield of $2^{m}\cdot\ST_m$-invariants in $L$. Clearly, $L/K$ is a finite
 Galois extension with Galois group $2^{m}\cdot\ST_m$. In particular, $\bar{L}=\bar{K}$.
Since $m \ge 5$, the only normal subgroups in $\ST_m$ are the subgroup $\{1\}$ of even index
$m!$, the alternating (sub)group $\AT_m$ of index $2$ and $\ST_m$ itself.

 The  even degree $2m$ polynomial
 $$h(x) =\prod_{i=1}^m (x^2-t_i^2)=\prod_{i=1}^m (x-t_i) \prod_{i=1}^m (x+t_i)$$
 lies in $K[x]$ and its splitting field coincides
 with $L$. It follows that $\Gal(h)=2^{m}\cdot\ST_m$.  Applying Theorem \ref{main} to the odd
 degree $(2m+1)$ polynomial
$$f(x):= x \cdot h(x)=x \cdot \prod_{i=1}^m (x^2-t_i^2),$$
 we conclude that the
 endomorphism ring (over $\bar{L}$) of the $m$-dimensional prymian $\mathrm{Prym}(C_{f,p})$ is canonically isomorphic to
   $\Z[\zeta_p]$.
\end{ex}

\begin{ex}
\label{magma} Suppose that an odd prime $p$ and an even integer $r$ enjoy one of the following properties.
\begin{enumerate}
\item[(1)]
$r \equiv 2 \bmod p$ (e.g., $r=2$, i.e., $m=2p-1$);
\item[(2)]
$r<\log_2(p-1)+2$ ;
\item[(3)]
 $1+2^{r-2}$ is not divisible by $p$.
\end{enumerate}
Let us put
$$K=\Q(\zeta_p), m=pr-1, \quad u(x)=x^m-x-1, \quad h(x)=u(x^2)=x^{2m}-x^2-1.$$
By Theorem \ref{x2mx2} (see below),
$$\Gal(f/K)=\Gal(h/\K)=\W(\D_{m}).$$
Applying Theorem \ref{main} to 
$$K=\Q(\zeta), \ f(x)=x h(x)=x^{2m+1}-x^3-x \in  K[x],$$
we conclude that 
 the complex $m(p-1)/2$-dimensional Prym variety $\mathrm{Prym}(C_{f,p})$ is  simple,  is not isomorphic to a jacobian (as an algebraic variety); in addition, its  endomorphism ring is canonically isomorphic to
   $\Z[\zeta_p]$.

\end{ex}

\begin{proof}[Proof of Theorem \ref{main0}]
The assertions (i) (except (i)(1) ) and (ii)(a) are already proven in Subsection \ref{Triprym} and
Remark \ref{multprime}.  Since $\ST_m$ is the doubly transitive permutation
group that does not contain normal subgroups of odd index (except $\ST_m$
itself) and $\W(\D_m)=2^{m-1}\cdot \ST_m$, the assertion (ii)(b1) follows from
Theorem \ref{main} applied to $\GG=\ST_m$.

In order to prove the assertion (i)(1), notice that (in the notation of Remark \ref{multprime})
$$\mathrm{mult}(\zeta_p^{-1})=r \cdot 1-1=r-1< \frac{1}{p} \cdot pr-\frac{p-1}{p}=
\frac{1}{p} \cdot \frac{2 \ \dim(\Prym(C_{f,p}))}{p-1}-\frac{p-1}{p}.$$
Now the assertion i)(1)  follows from the assertion (i)(2) combined with the
Theorem 2.17 of \cite{ZarhinAB}.

In order to prove the assertion (ii)(b2), notice that we already know (thanks
to the assertion (ii)(b1) ) that 
$$\End(\mathrm{Prym}(C_{f,p}))=\Z[\delta_p]\cong\Z[\zeta_p].$$
This implies that $\mathrm{Prym}(C_{f,p})$ is absolutely simple and its endomorphism ring  has exactly one Rosati involution that acts on 
$$\End(\mathrm{Prym}(C_{f,p})) =\Z[\zeta_p]$$
as the complex conjugation on the $p$th cyclotomic ring.. Clearly, this involution is 
$\delta_p$-invariant. So, if $\mathrm{Prym}(C_{f,p})$ is
isomorphic to the jacobian of a smooth connected projective curve then $\delta_p$ respects the corresponding canonical principal polarization attached to the curve.
Now the assertion (ii)(b2) follows from the assertion (i)(1)(C).

\end{proof}

\section{Appendix: Galois properties of certain polynomials}
\label{appen}

\begin{thm}
\label{x2mx2}
Let $m \ge 3$ be an odd integer. Then the following statements hold.
\begin{itemize}
\item[(i)]
 The polynomial
$$h(x)=x^{2m}-x^2-1$$ is irreducible over $\Q$ and its Galois group over $\Q$
is  $\W(\D_m)$.
\item[(ii)]
Suppose that $m=pr-1$ where $p$ is an odd prime and $r \ge 2$ is an integer.
Then the Galois group of $h(x)$ over the $p$th cyclotomic field $\Q(\zeta_p)$ is still $\W(\D_m)$ if $r$ enjoys one of the following properties.
\begin{enumerate}
\item[(1)]
$r \equiv 2 \bmod p$ (e.g., $r=2$, i.e., $m=2p-1$);
\item[(2)]
$r<\log_2(p-1)+2$ ;
\item[(3)]
 $1+2^{r-2}$ is not divisible by $p$.
\end{enumerate}
\end{itemize}

\end{thm}

In order to prove Theorem \ref{x2mx2}, we need the following assertion.

\begin{lem}
\label{irredx2m}
Let $m \ge 3$ be an odd integer, $u(x) \in \Z[x]$ a monic degree $m$ polynomial that enjoy the following properties.

\begin{itemize}
\item[(a)]
$u(x)$ is irreducible over $\Q$; in particular, its constant term is not zero.
\item[(b)]
$u(x)+(x+c) \in x^3 \Z[x]$ for some nonzero $c\in \Z$, i.e., the coefficient of $u(x)$ at $x^2$ is $0$ while the coefficient at  $x$ is $-1$.
\end{itemize}
Then $u(x^2)$ is irreducible over $\Q$.
\end{lem}

\begin{proof}[Proof of Lemma \ref{irredx2m}]
Clearly,  $u(x^2)$ is a degree $2m$ monic polynomial with integer coefficients, whose constant term is $-c$. More precisely, property (b)
implies that
\begin{equation}
\label{ux6}
u(x^2)- (x^2-c) \in x^6 \Z[x].
\end{equation}
Suppose that $u(x^2)$ is not irreducible. Since $u(x^2)$ is a monic polynomial with integer coefficients,  there are monic nonconstant polynomials 
$w_1(x), w_2(x) \in \Z[x]$ such that
$$u(x^2)=w_1(x) w_2(x);$$
in addition, $w_1(x)$ is irreducible over $\Q$. Let us put $d:=\deg(w_1) \ge 1$.  Since $u(x^2)$ is obviously even,
$$u(x^2)=w_1(-x) w_2(-x);  \quad  w_1(-x),   w_2(-x)\in \Z[x].$$
Clearly, $w_1(-x)$ is also an irreducible polynomial of degree $d$ over $\Q$, whose leading coefficient is $(-1)^d$.
Since the degree of monic $u(x^2)$ is $2m$, which is even, the leading coefficient of $w_2(-x)$ is also $(-1)^d$.
The irreducibility of $w_1(x)$ and $w_1(-x)$ implies that one of the following conditions holds.

\begin{itemize}
\item[(i)]
$w_1(x)$ and $w_1(-x)$ are relatively prime;
\item[(ii)]
$w_1(-x)=(-1)^d w_1(x).$
\end{itemize}

{\bf Case (ii)}  If $d$ is odd then $w_1(x)$ is an odd polynomial and therefore $w_1(0)=0$.
This implies that $0$ is a root of $u(x^2)$, which is wrong, since the constant term of $u(x^2)$ is $-c \ne 0$.
So, $d$ is even and therefore $w_1(x)$ is even. This implies that $w_2(x)$ is also even. This implies that there are monic
polynomials
$$v_1(x), \ v_2(x)\in \Z[x]$$
such that
$$w_1(x)=v_1(x^2), \ \  w_2(x)=v_2(x^2)$$
and therefore
$$\deg(v_1)=\frac{\deg(w_1)}{2}>0,  \quad \deg(v_2)=\frac{\deg(w_2)}{2}>0, \quad u(x^2)=v_1(x^2) v_2(x^2).$$
It follows that
$$u(x)=v_1(x) v_2(x),$$
which contradicts the irreducibility of $u(x)$.  So, Case (ii) does not hold.

{\bf Case (i)} 
Since both $w_1(x)$ and $w_1(-x)$ are divisors of $u(x^2)$, their product
$w_1(x) w_1(-x)$ divides $u(x^2)$ in $\Q[x]$.  Since
$w_1(x) (-1)^d w_1(-x)$ is a monic polynomial  with integer coefficients
that divides  $u(x^2)$ in $\Q[x]$, it also divides monic $u(x^2)$ in $\Z[x]$.
This means that there is a monic polynomial $v(x) \in \Z[x]$ such that
$$u(x^2)=v(x) w_1(x) (-1)^d w_1(-x)=(-1)^d w_1(x) w_1(-x) v(x).$$

It follows that $v(x)$ is an even polynomial of degree $2m-2d$ that divides $u(x^2)$.
Hence, there is a monic polynomial $v_{1/2}(x)\in \Z[x]$ of degree $m-d$ such that
$v(x)=v_{1/2}(x^2)$. This implies that $v_{1/2}(x)^2$ divides $u(x^2)$ in $\Q[x]$.
It follows that $v_{1/2}(x)$ divides  $u(x)$ in  $\Q[x]$. Since
$$\deg(v_{1/2})=m-d<m=\deg(u),$$
the irreducibility of $u(x)$ implies that $v_{1/2}(x)$ is a  constant. Since $v_{1/2}(x)$ is monic,
we get $v_{1/2}(x)=1$, i.e.,
$$u(x^2)=(-1)^d  w_1(-x) \cdot w_1(x).$$
This implies that
$$\deg(w_1)=\frac{2m}{2}=m.$$
It follows that
$$u(x^2)=(-1)^m  w_1(-x) \cdot w_1(x)=-w_1(-x) \cdot w_1(x)$$
(recall that $m$ is odd).  We need the notation for coefficients of $w_1(x)$:
$$w_1(x)=\sum_{i=0}^m a_i x^i \in \Z[x],  \quad a_i \in \Z.$$
Then
$$u(x^2)=-\left(\sum_{i=0}^m a_i (-x)^i \right)\left(\sum_{i=0}^m a_i x^i \right)=$$
$$\left(-a_0+a_1x-a_2 x^2+a_3 x^3-a_4 x^4 +\dots\right) 
\left(a_0+a_1 x+a_2 x^2+a_3 x^3+a_4 x^4+\dots\right).$$
Combining it with \eqref{ux6}, we get the equality
$$-x^2-c=-a_0^2+\left(a_1^2-2 a_0 a_2\right)x^2+\left(-2a_4a_0+2a_1a_3-a_2^2\right)x^4$$
in $\Z[x]$. This means that
\begin{equation}
\label{a0a1a27}
-a_0^2=-c, \quad a_1^2-2 a_0 a_2=-1, \quad -2a_4a_6+2a_1a_3-a_2^2=0.
\end{equation}
First equality of \eqref{a0a1a27} implies that $a_0 \ne 0$.
Second equality of \eqref{a0a1a27} implies that $a_1$ is odd ,
i.e., $a_1=2b+1$ with $b \in \Z$. In addition,
$$a_2=\frac{a_1^2+1}{2 a_0}= \frac{(2b+1)^2+1}{2 a_0}=
 \frac{2b^2+2b+1}{a_0}.$$
 Hence, $a_2$ is a divisor of the odd integer $2b^2+2b+1$ and therefore is also odd.
However, third equality of \eqref{a0a1a2} implies that $a_2$ is even. The obtained contradiction
proves that Case II does not hold, which ends the proof.
\end{proof}

\begin{rem}
\label{DeltaC}
\begin{itemize}
\item[(i)]
The discriminant of a trinomial $u(x)=x^m-x-c$ (with odd integers $m>1$ and $c \in \Z$) is
$$\Delta_c:=(-1)^{m(m-1)/2} m^m (-c)^{m-1}+(-1)^{(m-1)(m-2)/2} (m-1)^{m-1}(-1)^m=$$
$$(-1)^{(m^2+m)/2} c^{m-1}+(-1)^{(m^2-m+2)/2}(m-1)^{m-1}.$$
(see \cite[Problem 835]{FS}).  Since $m$ is odd, the integers $(m^2+m)/2$ and 
$m^2-m+2)/2$ have the same parity (their difference $m-1$ is an even integer).
This implies that
$$\Delta_c=\pm \left(c^{m-1}+(m-1)^{m-1}\right).$$
In particular, $\Delta_c$ is an {\sl odd integer}, because $c$ is odd and $(m-1)$ is even. 

On the other hand, the discriminant $\Delta^{(2)}_c$ of the trinomial
$$u(x^2)=x^{2m}-x^2-c$$
can be explicitly computed by the formula 
\begin{equation}
\label{discUx2}
\Delta^{(2)}_c=2^m \Delta_c^2 \cdot 1 \cdot (-1)=-2^m \Delta_c^2=-2^m  \left(c^{m-1}+(m-1)^{m-1}\right)^2
\end{equation}
(see \cite[Problem 853]{FS}).

 For example, if $c=1$ then
$$\Delta_1=(-1)^{(m^2+m)/2} \cdot 1^{m-1}+(-1)^{(m^2-m+2)/2}(m-1)^{m-1}=
(-1)^{(m^2+m)/2}+(-1)^{(m^2-m+2)/2}(m-1)^{m-1}.$$
\item[(ii)]
Suppose that that $m=pr-1$ where $p$ is an odd prime and $r \ge 2$ is an even  integer.  Then
$m-1 \equiv (-2) \bmod p$.

 If $c=1$ then
$$\Delta_1=\pm \left(1+(pr-2)^{pr-2}\right) \equiv \pm \left(1+2^{p(r-1)+(p-2)}\right) \bmod p$$
$$\equiv  \pm \left(1+2^{r-1+p-2}\right) \bmod p \equiv  \pm \left(1+2^ {p-1+r-2}\right) \bmod p
\equiv \pm \left(1+2^{r-2}\right) \bmod p.$$
This  says that $\Delta_1$ is not divisible by prime if and only if $1+2^{r-2}$ is not divisible by $p$.
In light of \eqref{discUx2}, $\Delta^{(2)}_1$ is not divisible by prime if and only if $1+2^{r-2}$ is not divisible by $p$.

In particular, if $r\equiv 2 \bmod (p-1)$ (e.g., $r=2$ and $m=2p-1$) then both $\Delta_1$ and $\Delta^{(2)}_1$ are {\sl not} divisible by $p$,
because in this case $r-2=k(p-1)$ for some nonnegative integer $k$ and
$$1+2^{r-2}=1+2^{p(k-1)} \equiv (1+1^{k-1})\bmod p =2 \bmod p \ne 0 \bmod p,$$ 
 since $p$ is an {\sl odd} prime. 
 
 Another example of the non-divisibility of the discriminants  
is provided by {\sl small} $r$. Namely, if $r<\log_2(p-1)+2$ then the integer
$1+2^{r-2}$ is srictly grieater than $1+2^{\log_2(p-1)}=p$ and therefore is not divisible by $p$.
This implies that if $r$ is an even integer such that
$$r<\log_2(p-1)+2$$
then both $\Delta_1$ and $\Delta^{(2)}_1$ are {\sl not} divisible by $p$.
\end{itemize}
\end{rem}

\begin{lem}
\label{IrredDelta}
Let $m \ge 7$ be an odd integer. Let $c$ be an odd integer such that that the trinomial
$u(x):=x^m-x-c$ is irreducible over $\Q$ and its Galois group is $\mathbf{S}_m$.
Then
 the trinomial $x^{2m}-x^2-c$  is irreducible over the quadratic field $\Q(\sqrt{\Delta_c})$.
\end{lem}

\begin{proof}[Proof of Lemma \ref{IrredDelta}]
Let us put $u(x):=x^m-x-c$. Then
$$u(x^2)=x^{2m}-x^2-c.$$
Clearly, the Galois group of  $u(x)$ over the quadratic field $K:= \Q(\sqrt{\Delta_c})$ is $\mathbf{A}_m$, which is a transitive subgroup of $\mathbf{S}_m$. Hence, $u(x)$
 remains irreducible  over $K$.

It follows from Remark \ref{DeltaC} that the quadratic extension  field $K/\Q$
is unramified at $2$. This allows us to mimick arguments of Proof of Lemma \ref{irredx2m}]
in the ring $\mathcal{O}$ of integers in $K$ instead of $\Z$. The unramifiedness of $2$
implies that if $A \in \mathcal{O}$ satisfies $A^2 \in 2  \mathcal{O}$ then $A \in  \mathcal{O}$.

Suppose that $u(x^2)$ is not irreducible in $K[x]$. This means that there  are monic nonconstant polynomials 
$w_1(x), w_2(x) \in K[x]$ such that
$$u(x^2)=w_1(x) w_2(x);$$
in addition, $w_1(x)$ is irreducible over $K$.  Since $u(x^2) \in \mathcal{O}[x]$ and 
$\mathcal{O}$ is an integrally closed domain,
$$w_1(x), w_2(x) \in \mathcal{O}[x].$$
Let us put $d:=\deg(w_1) \ge 1$.  Since $u(x^2)$ is obviously even,
$$u(x^2)=w_1(-x) w_2(-x);  \quad  w_1(-x),   w_2(-x)\in \mathcal{O}[x].$$
Clearly, $w_1(-x)$ is also an irreducible polynomial of degree $d$ over $K$ but its  leading coefficient is $(-1)^d$.
Since the degree of monic $u(x^2)$ is $2m$, which is even, the leading coefficient of $w_2(-x)$ is also $(-1)^d$.
The irreducibility of  both $w_1(x)$ and $w_1(-x)$ implies that one of the following conditions holds.

\begin{itemize}
\item[(i)]
$w_1(x)$ and $w_1(-x)$ are relatively prime;
\item[(ii)]
$w_1(-x)=(-1)^d w_1(x).$
\end{itemize}

{\bf Case (ii)}  If $d$ is odd then $w_1(x)$ is an odd polynomial and therefore $w_1(0)=0$.
This implies that $0$ is a root of $u(x^2)$, which is wrong, since the constant term of $u(x^2)$ is $-1 \ne 0$.
So, $d$ is even and therefore $w_1(x)$ is even. This implies that $w_2(x)$ is also even. It follows that there are monic
polynomials
$$v_1(x), \ v_2(x)\in \mathcal{O}[x]$$
such that
$$w_1(x)=v_1(x^2), \ \  w_2(x)=v_2(x^2).$$
Hence,
$$\deg(v_1)=\frac{\deg(w_1)}{2}>0,  \quad \deg(v_2)=\frac{\deg(w_2)}{2}>0, \quad u(x^2)=v_1(x^2) v_2(x^2).$$
It follows that
$$u(x)=v_1(x) v_2(x),$$
which contradicts the irreducibility of $u(x)$ over $K$.  So, Case (ii) does not hold.

{\bf Case (i)} 
Since both $w_1(x)$ and $w_1(-x)$ are divisors of $u(x^2)$, their product
$w_1(x) w_1(-x)$ divides $u(x^2)$ in $K[x]$.  Since
$w_1(x) (-1)^d w_1(-x)$ is a monic polynomial  with  coefficients in $\mathcal{O}$ 
that divides  $u(x^2)$ in $K[x]$, it also divides monic $u(x^2)$ in $\mathcal{O}[x]$.
This means that there is a monic polynomial $v(x) \in \mathcal{O}[x]$ such that
$$u(x^2)=v(x) w_1(x) (-1)^d w_1(-x)=(-1)^d w_1(x) w_1(-x) v(x).$$

It follows that $v(x)$ is an even polynomial of degree $2m-2d$ that divides $u(x^2)$.
Hence, there is a monic polynomial $v_{1/2}(x)\in \Z[x]$ of degree $m-d$ such that
$v(x)=v_{1/2}(x)^2$. This implies that $v_{1/2}(x)^2$ divides $u(x^2)$ in $\Q[x]$.
It follows that $v_{1/2}(x)$ divides  $u(x)$ in  $\Q[x]$. Since
$$\deg(v_{1/2})=m-d<m=\deg(u),$$
the irreducibility of $u(x)$ implies that $v_{1/2}(x)$ is a  constant. Since $v_{1/2}(x)$ is monic,
we get $v_{1/2}(x)=1$, i.e.,
$$u(x^2)=(-1)^d  w_1(-x) \cdot w_1(x).$$
This implies that
$$\deg(w_1)=\frac{2m}{2}=m.$$
It follows that
$$u(x^2)=(-1)^m  w_1(-x) \cdot w_1(x)=-w_1(-x) \cdot w_1(x)$$
(recall that $m$ is odd).  We need the notation for coefficients of $w_1(x)$:
$$w_1(x)=\sum_{i=0}^m a_i x^i \in \mathcal{O}[x],  \quad a_i \in \mathcal{O}.$$
Then
$$u(x^2)=-\left(\sum_{i=0}^m a_i (-x)^i \right)\left(\sum_{i=0}^m a_i x^i \right)=$$
$$\left(-a_0+a_1x-a_2 x^2+a_3 x^3-a_4 x^4 +\dots\right) 
\left(a_0+a_1 x+a_2 x^2+a_3 x^3+a_4 x^4+\dots\right).$$
Combining it with \eqref{ux6}, we get the equality
$$-x^2-c=-a_0^2+\left(a_1^2-2 a_0 a_2\right)x^2+\left(-2a_4a_0+2a_1a_3-a_2^2\right)x^4$$
in $\mathcal{O}[x]$. This means that
\begin{equation}
\label{a0a1a2}
-a_0^2=-c, \quad a_1^2-2 a_0 a_2=-1, \quad -2a_4a_6+2a_1a_3-a_2^2=0.
\end{equation}
First equality of \eqref{a0a1a2} implies that 
$a_0 \not\in 2 \mathcal{O}$.  
Second equality of \eqref{a0a1a2} implies that 
$a_1=2b+1$ with $b \in \mathcal{O}$. In addition,
$$a_2=\frac{a_1^2+1}{2 a_0}= \frac{(2b+1)^2+1}{2 a_0}
=  \frac{2b^2+2b+1}{a_0}
 \not\in 2 \mathcal{O}.$$
However, third equality of \eqref{a0a1a2} implies that $a_2^2\in 2\mathcal{O}$ and therefore $a_2\in 2\mathcal{O}$.
The obtained contradiction
proves that Case II does not hold, which ends the proof.
\end{proof}

\begin{lem}
\label{squareRoot}
Let $m \ge 9$ be an odd integer. Let $c$ be an odd integer such that that the trinomial
$u(x):=x^m-x-c$ is irreducible over $\Q$ and its Galois group is $\mathbf{S}_m$.
 
 Let  $\mathfrak{R}_u \subset \bar{\Q}$ be the $m$-element set of roots of $u(x)$ and $\Q(\mathfrak{R}_u)\subset \bar{\Q}$ be
 the splitting field of $u(x)$ over $\Q$.  Then:
 \begin{enumerate}
 \item[(i)]
 $\Q(\sqrt{\Delta_c})\subset \Q(\mathfrak{R}_u)$;
 \item[(ii)]
 If $\alpha \in \mathfrak{R}_u \subset \Q(\mathfrak{R}_u)$ is a root of $u(x)$ then $\sqrt{\alpha}\not\in \Q(\mathfrak{R}_u)$.
 \end{enumerate}
\end{lem}

\begin{proof}
(i) is obvious. It is also clear that $ \Q(\mathfrak{R}_u)$ is  a splitting field of $u(x)$ over $\Q(\sqrt{\Delta_c})$
and the Galois group of the normal  field extension $\Q(\mathfrak{R}_u)/\Q(\sqrt{\Delta_c})$ is the alternating group 
$\Alt((\mathfrak{R}_u) \cong \mathbf{A}_m$.

Notice that the $2m$-element set 
$$\mathfrak{R}_h=\{\beta \in \bar{\Q}\mid \beta^2 \in \mathfrak{R}_u)\}$$
is the set of roots of the polynomial
$$h(x)=u(x^2)=x^{2m}-x^2-c.$$
Taking into account that $h(x)$ is irreducible over $\Q(\sqrt{\Delta_c})$ while the field extension $\Q(\mathfrak{R}_u)/\Q(\sqrt{\Delta_c})$
is Galois, we conclude that if $\sqrt{\beta}\in \Q(\mathfrak{R}_u)$ for some $\beta \in \mathfrak{R}_u$ then the whole
$\mathfrak{R}_h$ lies in $\Q(\mathfrak{R}_u)$ and
$$\Alt((\mathfrak{R}_u) =\Gal(\Q(\mathfrak{R}_u)/\Q(\sqrt{\Delta_c}))$$
acts transitibely on the $2m$-element set $\mathfrak{R}_h$. It follows that
$\Alt((\mathfrak{R}_u)$ has a subgroup of index $2m$, i.e.,  $\mathbf{A}_m$ has a subgroup of index $2m$,
which contradicts to our assumptions on $m$ (see \cite[Th. 5.2A and Remark on p. 147]{DixonMortimer} applied to $r=2$).
\end{proof}

Let $h(x)=u(x^2)$ be as in the proof of Lemma \ref{squareRoot} and 
$\Q(\mathfrak{R}_h)$ be the splitting field of  $h(x)$ over $\Q$. Clearly,
$\Q(\mathfrak{R}_h)$ contains $\Q(\mathfrak{R}_u)$ and $\Q(\mathfrak{R}_h)/\Q(\mathfrak{R}_u)$ 
is a Galois extension that is obtained by extracting square roots fron $m$ elements $\alpha \in \mathfrak{R}_u$.
This implies that that  $\Gal(\Q(\mathfrak{R}_h)/\Q(\mathfrak{R}_u))$  is an elementary commutative abelian $2$-group,
whose order divides $2^m$.

\begin{lem}
\label{2group}
Let $m \ge 9$ be an odd integer. Let $c$ be an odd integer that is a square in $\Q$ and such that that the trinomial
$u(x):=x^m-x-c$ is irreducible over $\Q$ and its Galois group is $\mathbf{S}_m$.


Then  $\Gal(h/\Q)=\Gal(\Q(\mathfrak{R}_h)/\Q)=\W(D_m)$
\end{lem}

\begin{proof}
The assertion follows Lemma \ref{evenR} applied to $K=\Q$.
\end{proof}

\begin{proof}[Proof of Theorem \ref{x2mx2}]

If $m \ge 9$ then the assertion (i) from Lemma \ref{2group} applied to $c=1$.
If $m=3,5,7$ then the result follows from computations with MAGMA, see \cite{MAGMA}.
(The case $m=5$ was already discussed in \cite{ZarhinI}.)

Let us prove the assertion (ii). 
It follows from Remark \ref{DeltaC} that under our assumptions on $p$ the discriminants of $\Q(\RR_h)$ and $\Q(\zeta_p)$ are relatively prime
and therefore the fields  $\Q(\RR_h)$ and $\Q(\zeta_p)$ are linearly disjoint over $\Q$. It follows that
the Galois group of $h(x)$ over $\Q(\zeta_p)$ is the same as over $\Q$. This ends the proof.

\end{proof}

\end{document}